\documentclass[article]{amsart}

 \usepackage{amssymb}

 \usepackage{url}

\usepackage{graphicx, amscd, verbatim, subfig, caption, epstopdf}
\usepackage{cite}

\newtheorem{thm}{Theorem}[section]
\newtheorem{cor}[thm]{Corollary}

\theoremstyle{remark}
\newtheorem{rem}[thm]{Remark}
\newtheorem{con}[thm]{Conclusion} 

\theoremstyle{definition}
\newtheorem{dfn}[thm]{Definition}
\newtheorem{ex}[thm]{Example}
\newtheorem{prb}[thm]{Open Problem}

\numberwithin{equation}{section}
\numberwithin{thm}{section}

\begin{document}






\title{New Approach to a generalized Fractional Integral}


%

\author{Udita Katugampola}
\address{southern illinois university, Carbondale, IL 62901, USA.
}

\email{uditanalin@yahoo.com}


%

\thanks{Submitted to : Applied Mathematics and Computation}

\thanks{Author highly appreciates the valuable comments and suggestions given by Dr. Jerzy Kocik, Department of mathematics, Southern Illinois University, Carbondale, IL 62901}

%

\begin{abstract}
The paper presents a new fractional integration, which generalizes the Riemann-Liouville and Hadamard fractional integrals into a single form, which when a parameter fixed at different values, produces the above integrals as special cases. Conditions are given for such a generalized fractional integration operator to be bounded in an extended Lebesgue measurable space. Semigroup property for the above operator is also proved. Finally, we give a general definition of the Fractional derivatives and give some examples.

\vspace{.4cm}
\noindent Submitted to : Applied Mathematics and Computation
\end{abstract}


\maketitle


 \section{Introduction}
History of Fractional Calculus (FC) goes back to seventeenth century, when in 1695 the derivative of order $\alpha = 1/2$ was described by Leibnitz. Since then, the new theory turned out to be very attractive to mathematicians as well as biologists, chemists, economists, engineers and physicsts. Several books were written on the theories and developments of FC \cite{key-8, key-2, key-9,key-11,key-12}. In \cite{key-9} Samko, et al, provides a comprehensive study of the subject. Several different derivatives were introduced: Riemann-Liouville, Hadamard, Grunwald-Letnikov, Riesz and Caputo are just a few \cite{key-8, key-2, key-9,key-11,key-12}.  

In fractional calculus, the fractional derivatives are defined by a fractional integral \cite{key-8, key-2, key-9, key-11, key-12}. There are several known forms of the fractional integrals of which two have been studied extensively for their applications \cite{key-1,key-2,key-3,key-5,key-6,key-7,key-8,key-9}. The first is the \emph{Riemann-Liouville fractional integral} defined for parameter $\alpha \in \textbf{R}$ by,
\begin{equation}
{}_a \mathbb{I}^\alpha_x f(x) = \frac{1}{\Gamma(\alpha)}\int_a^x (x - \tau)^{\alpha -1} f(\tau) d\tau  \quad ;\alpha > 0,\, x > a.
\label{eq:R-L}
\end{equation}
motivated by the \emph{Cauchy integral formula} 
\begin{equation}\label{eq:ch}
\int_a^x d\tau_1 \int_a^{\tau_1} d\tau_2 \cdots \int_a^{\tau_{n -1}} f(\tau_n)d\tau_n 
                = \frac{1}{\Gamma(n)}\int_a^x (x - \tau)^{n -1} f(\tau) d\tau 
\end{equation}
well-defined for $\alpha \in \mathbb{N}$. 
The second is the \emph{Hadamard Fractional integral} introduced by J. Hadamard \cite{key-10} ,  and given by,
\begin{equation}
{}_a \mathbf{I}^\alpha_x f(x) = \frac{1}{\Gamma(\alpha)}\int_a^x \Big(\log\frac{x}{\tau}\Big)^{\alpha -1} f(\tau)\frac{d\tau}{\tau} \quad ;\alpha > 0, \, x > a.
\label{eq:H}
\end{equation}
This is based on the generalization of the integral
\begin{equation}\label{eq:chh}
\int_a^x \frac{d\tau_1}{\tau_1} \int_a^{\tau_1} \frac{d\tau_2}{\tau_2} \cdots \int_a^{\tau_{n -1}} \frac{f(\tau_n)}{\tau_n}\, d\tau_n 
                = \frac{1}{\Gamma(n)}\int_a^x \Big(\log \frac{x}{\tau}\Big)^{n -1} f(\tau) \frac{d\tau}{\tau} 
\end{equation}

Here we want to present the fractional integration, which generalizes both the Riemann-Liouville and Hadamard fractional integrals into a single form. New generalization is based on the observation that, for $n \in \mathbb{N}$,
\begin{align}
\int_a^x \tau_1^\rho d\tau_1 \int_a^{\tau_1} \tau_2^\rho d\tau_2 \cdots &\int_a^{\tau_{n -1}} \tau_n^\rho f(\tau_n)d\tau_n \nonumber \\ 
&= \frac{(\rho +1)^{1-n}}{(n-1)!} \int_a^x (t^{\rho + 1}-\tau^{\rho +1})^{n-1} \tau^\rho f(\tau) d\tau
\label{eq:nint1}
\end{align}
which suggests the following fractional version
\begin{equation}
{}^\rho_a I^\alpha_x f(x) = \frac{(\rho + 1)^{1-\alpha }}{\Gamma({\alpha})} \int^x_a (x^{\rho +1} - \tau^{\rho +1})^{\alpha -1} \tau^\rho f(\tau) d\tau 
\label{eq:fint1}
\end{equation}
where $\alpha \; \text{and} \, \rho \neq -1$ are real numbers.

These integrals are called \emph{left-sided} integrals. Similarly we can define \emph{right-sided} integrals \cite{key-8, key-2, key-9}. 

In the subsequent sections, we give conditions for the integration operator ${}^\rho_a I^\alpha_x$ to be bounded in $\textit{X}^p_c(a,b)$. We also establish semigroup property for the generalized fractional integration operator and finally, we give the general definition of the Fractional derivatives. 


\section {Generalization of the fractional integration}
Consider the space $\textit{X}^p_c(a,b) \; (c\in \textbf{R}, \, 1 \leq p \leq \infty)$ of those complex-valued Lebesgue measurable functions $f$ on $[a, b]$ for which $\|f\|_{\textit{X}^p_c} < \infty$, where the norm is defined by
\begin{equation}\label{eq:df1}
\|f\|_{\textit{X}^p_c} =\Big(\int^b_a |t^c f(t)|^p \frac{dt}{t}\Big)^{1/p} < \infty \quad (1 \leq p < \infty,\, c \in \textbf{R}) 
\end{equation}
\noindent and for the case $p=\infty$
\begin{equation} \label{eq:df2}
\|f\|_{\textit{X}^\infty_c} = \text{ess sup}_{a \leq t \leq b} [t^c|f(t)|]  \quad ( c \in \textbf{R}).
\end{equation}
In particular, when $c = 1/p \; (1 \leq p \leq \infty),$ the space $\textit{X}^p_c$ coincides with the classical $\textit{L}^p(a,b)$-space with
\begin{align} \label{eq:df3}
&\|f\|_p =\Big(\int^b_a |f(t)|^p \frac{dt}{t}\Big)^{1/p} < \infty \quad (1 \leq p < \infty), \\
&\|f\|_\infty = \text{ess sup}_{a \leq t \leq b} |f(t)|   \quad ( c \in \textbf{R}).
\end{align}

To derive a formula for the generalized integral, consider for natural $n \in \mathbb{N} = \{1, 2, \dots \}$ and real $\rho$ and $ a \geq 0$, the n-fold integral of the form 
\begin{equation}
{}^\rho_a \mathcal{I}^\alpha_x f(x) = \int_a^x \tau_1^\rho d\tau_1 \int_a^{\tau_1} \tau_2^\rho d\tau_2 \cdots \int_a^{\tau_{n -1}} \tau_n^\rho f(\tau_n)d\tau_n 
\label{eq:inte-g}
\end{equation}

\noindent First notice that-using the Dirichlet technique (see p. 64 of \cite{key-2}), we have

\begin{align*}
\int_a^x \tau_1^\rho d\tau_1 \int_a^{\tau_1} \tau^\rho f(\tau) d\tau &= \int_a^x \tau^\rho f(\tau) d\tau \int_\tau^x \tau_1^\rho d\tau_1  \\ &= \frac{1}{\rho +1} \int_a^x (x^{\rho + 1}-\tau^{\rho +1}) \tau^\rho f(\tau) d\tau 
\end{align*}
Repeating the above step $n-1$ times we arrived at
\begin{align}
\int_a^x \tau_1^\rho d\tau_1 \int_a^{\tau_1} \tau_2^\rho d\tau_2 \cdots &\int_a^{\tau_{n -1}} \tau_n^\rho f(\tau_n)d\tau_n \nonumber \\ 
&= \frac{(\rho +1)^{1-n}}{(n-1)!} \int_a^x (x^{\rho + 1}-\tau^{\rho +1})^{n-1} \tau^\rho f(\tau) d\tau
\label{eq:nint}
\end{align}
Fractional version of (\ref{eq:nint}) is the following
\begin{equation}
{}^\rho_a I^\alpha_x f(x) = \frac{(\rho + 1)^{1-\alpha}}{\Gamma({\alpha})} \int^x_a (x^{\rho +1} - \tau^{\rho +1})^{\alpha -1} \tau^\rho f(\tau) d\tau 
\label{eq:fint}
\end{equation}
where $\alpha$ and $\rho \neq -1$ are real numbers. When $\rho = 0$ we arrive at the standard Riemann-Liouville fractional integral, which is used to define both the Riemann-Liouville and Caputo fractional derivatives \cite{key-8, key-2, key-9}. Using L'hospital rule, when $\rho \rightarrow -1^+$, we have
\begin{align*}
\lim_{\rho \rightarrow -1^+} \frac{(\rho + 1)^{1-\alpha}}{\Gamma({\alpha})} & \int^x_a (x^{\rho +1} - \tau^{\rho +1})^{\alpha -1} \tau^\rho f(\tau) d\tau \\
&=\frac{1}{\Gamma({\alpha})}  \int^x_a \lim_{\rho \rightarrow -1^+} \bigg(\frac{x^{\rho +1} - \tau^{\rho +1}}{\rho + 1}\bigg)^{\alpha -1} \tau^\rho f(\tau) d\tau \\
&= \frac{1}{\Gamma(\alpha)} \int_a^x \Big(\log \frac{x}{\tau} \Big)^{\alpha -1} \frac{f(\tau)}{\tau} d\tau
\end{align*}
This is the famous Hadamard fractional integral (\ref{eq:H}) ~\cite{key-1,key-2,key-3,key-5,key-6,key-7,key-8}. Hadamard fractional integral has been extensively studied, especially, Hadamard-type fractional calculus \cite{key-3}, composition and semigroup properties \cite{key-1}, Mellin transforms \cite{key-5}, integration by parts formulae \cite{key-6}, and \emph{G-transform} representations \cite{key-7} are just a few.   

\section{Boundedness in the space $\textit{X}^p_c(a,b)$}\label{sec: gfi}
In this section we show that the generalized fractional integration operator $ {}^\rho_a I^\alpha_t$ is well-defined on $\textit{X}^p_c(a,b)$ for $\rho \geq c $. We have the following theorem

\begin{thm} \label{eq:th1}
Let $\alpha > 0,\, 1 \leq p \leq \infty,\, 0 <a < b < \infty$ and let $\rho \in \textbf{R}$ and $c \in \textbf{R}$ be such that $\rho \geq c$. Then the operator ${}^\rho_a I^\alpha_t$ is bounded in $\textit{X}^p_c(a,b)$ and 
\begin{equation}
\|{}^\rho_a I^\alpha_{t}f \|_{\textit{X}^p_c} \leq \sl{K}\|f\|_{\textit{X}^p_c}
\end{equation}
\label{eq:thm1}
where 
\begin{equation}
\sl{K} = \frac{b^{\alpha(\rho+1)-1}}{\Gamma(\alpha)}\int^{\frac{b}{a}}_1 u^{c-\alpha(\rho+1)-1}\Big(\frac{u^{\rho+1}-1}{\rho +1}\Big)^{\alpha -1} du \quad \text{; $\rho \neq -1$}
\end{equation}
\label{eq:con1}


\end{thm}

\begin{proof}
First consider the case $1 \leq p \leq \infty$. Since $f(t) \in \textit{X}^p_c(a,b)$, then $t^{c-1/p}f(t) \in \textit{L}_p(a,b)$ and we can apply the generalized Minkowsky inequality. We have
\begin{align}
\|{}^\rho_{a}I^\alpha_{t}f\|_{\textit{X}^p_c}
&=\Big(\int_a^{b}x^{cp}\left|\frac{1}{(\rho+1)^{\alpha-1}\Gamma(\alpha)}\int^x_a\big(x^{\rho+1}-t^{\rho+1}\big)^{\alpha-1}t^\rho f(t) dt\right|^p\frac{dx}{x} \Big)^\frac{1}{p} \nonumber \\
&=\Big(\int_a^{b}\left|\frac{1}{\Gamma(\alpha)}\int^x_a x^{c-\frac{1}{p}}\, t^\rho \Big(\frac{x^{\rho+1}-t^{\rho+1}}{\rho +1}\Big)^{\alpha-1} f(t) dt\right|^p dx \Big)^\frac{1}{p}\nonumber \\
&=\Big(\int_a^{b}\left|\frac{1}{\Gamma(\alpha)}\int^x_a x^{c-\frac{1}{p}}\, t^{\alpha(\rho+1)-1} \Big(\frac{{(\frac{x}{t})}^{\rho+1}-1}{\rho +1}\Big)^{\alpha-1} f(t) dt\right|^p dx \Big)^\frac{1}{p}\nonumber \\
&=\Big(\int_a^{b}\left|\frac{1}{\Gamma(\alpha)}\int_1^\frac{x}{a} x^{c-\frac{1}{p}}\, \Big(\frac{x}{u}\Big)^{\alpha(\rho+1)-1} \Big(\frac{u^{\rho+1}-1}{\rho +1}\Big)^{\alpha-1} f\Big(\frac{x}{u}\Big) x \frac{du}{u^2}\right|^p dx \Big)^\frac{1}{p}\nonumber \\
&\leq \int_1^{\frac{b}{a}}\frac{1}{\Gamma(\alpha)}\Big(\frac{u^{\rho+1}-1}{\rho +1}\Big)^{\alpha -1}\cdot\frac{1}{u^{\alpha(\rho+1)}}\Big(\int_{at}^b x^{cp} \big|f\Big(\frac{x}{u}\Big)\big|^p\,\frac{dx}{x}\Big)^{\frac{1}{p}} du \cdot b^{\alpha(\rho+1)-1}\nonumber \\
 &= \int_1^{\frac{b}{a}}\frac{b^{\alpha(\rho+1)-1}}{\Gamma(\alpha)}\Big(\frac{u^{\rho+1}-1}{\rho +1}\Big)^{\alpha -1}\cdot\frac{u^c}{u^{\alpha(\rho+1)+1}}\Big(\int_a^{b/u} \big|t^c f(t)\big|^p\,\frac{dt}{t}\Big)^{\frac{1}{p}}du \nonumber 
\end{align}
and hence 
\begin{equation*}
\|{}^\rho_a I^\alpha_{t}f \|_{\textit{X}^p_c} \leq \sl{M}\|f\|_{\textit{X}^p_c}
\end{equation*}
where
\begin{equation}
   \sl{M} = \frac{b^{\alpha(\rho+1)-1}}{\Gamma(\alpha)}\int^{\frac{b}{a}}_1 u^{c-\alpha(\rho+1)-1}\Big(\frac{u^{\rho+1}-1}{\rho +1}\Big)^{\alpha -1} du \qquad  ; 1\leq p < \infty
\label{eq:con2}   
\end{equation}
thus, Theorem \ref{eq:th1} is proved for $1\leq p < \infty$. For $p=\infty$, by (\ref{eq:df2}) and (\ref{eq:fint}) we have
\begin{align}
\Big|x^c\big({}^\rho_a I^\alpha_{t}f\big)(x)\Big| &\leq \frac{1}{(\rho + 1)^{\alpha -1}\Gamma({\alpha})} \int^x_a (x^{\rho +1} - \tau^{\rho +1})^{\alpha -1} \tau^\rho \Big(\frac{x}{\tau}\Big)^c \big|\tau^cf(\tau)\big| d\tau \nonumber \\
&=\frac{b^{\alpha(\rho+1)-1}}{\Gamma(\alpha)}\int^{\frac{b}{a}}_1 u^{c-\alpha(\rho+1)-1}\Big(\frac{u^{\rho+1}-1}{\rho +1}\Big)^{\alpha -1} du \nonumber
\end{align}
after the substitution $u=x/\tau$. This agrees with (\ref{eq:con2}) above. This completes the proof of the theorem.  
\end{proof}	

\begin{rem}
Note that this proof is similar to the proof of the Theorem 2.1 in \cite{key-3}.
\end{rem}

The following version of the Theorem \ref{eq:th1} has been proved in \cite{key-3}, for the special case when $\rho \rightarrow -1^+$.  


\begin{thm}
Let $\alpha >0,\, 1 \leq p \leq \infty,\, 0 <a <b <\infty$ and let $\rho \in \textbf{R}$ and $c\in \textbf{R}$ be such that $\rho \geq c$. Then the operator ${}^\rho_{a} I^{\alpha}_{t}$ is bounded in $\textit{X}^p_c(a,b)$ and
\begin{equation}
\|{}^\rho_a J^\alpha_{t}f \|_{\textit{X}^p_c} \leq \mathcal{K}\|f\|_{\textit{X}^p_c}
\end{equation} 
where 
\begin{equation*}
\mathcal{K} = \frac{1}{\Gamma(\alpha+1)}\Big(\log\frac{b}{a}\Big)^\alpha
\end{equation*}
for $\rho = c$, while
\begin{equation*}
\mathcal{K} = \frac{1}{\Gamma(\alpha)}(\rho - c)^{-\alpha} \gamma\Big[\alpha,(\rho - c)\log\Big(\frac{b}{a}\Big)\Big]  
\end{equation*}
for $\rho >c$, where $\gamma(\alpha, \beta)$ is the incomplete gamma-function defined by
\begin{equation*}
\gamma(\alpha, \beta)= \int^x_0 t^{\alpha -1}e^{-t}\,dt
\end{equation*}
\end{thm}

Substituting $c=1/p$ in Theorem \ref{eq:th1} and taking (\ref{eq:df3}) into account, we deduce the boundedness of the operator ${}^\rho_a I^\alpha_{t}$ in the space $\textit{L}^p (a,b)$. 

\begin{cor}
Let $\alpha >0, \, 1 \leq p \leq \infty,\, 0 <a <b <\infty$ and let $\rho \in \textbf{R}$ be such that $\rho \geq 1/p$. Then the operator ${}^\rho_{a} I^{\alpha}_{t}$ is bounded in $\textit{L}^p(a,b)$ and
\begin{equation}
\|{}^\rho_a J^\alpha_{t}f \|_p \leq \mathcal{K_1}\|f\|_p,
\end{equation} 
where 
\begin{equation}
\mathcal{K_1} = \frac{b^{\alpha(\rho+1)}}{\Gamma(\alpha)}\int^{\frac{b}{a}}_1 u^{\frac{1}{p}-\alpha(\rho+1)-1}\Big(\frac{u^{\rho+1}-1}{\rho +1}\Big)^{\alpha -1} du \quad \text{; $\rho \neq -1$}
\end{equation}
\label{eq:con3}
\end{cor}
Notice that this is Corollary 2.2 of \cite{key-3}. We now turn to algebraic properties of the integral operator. 

\section{Semigroup property}
In this section we give semigroup properties of the integral operator

\begin{thm} \label{eq:th2}
Let $\alpha >0,\, \beta >0,\, 1 \leq p \leq \infty, \, 0 < a < b < \infty $ and let $\rho \in \mathbb{R}$ and $c \in \mathbb{R}$ be such that $\rho \geq c$. Then for $f \in \textit{X}^p_c(a,b)$ the semigroup property holds. That is,
\begin{equation}
  {}^\rho_{a} I^{\alpha}_{t}\,{}^\rho_{a} I^{\beta}_{t}f = {}^\rho_{a} I^{\alpha+\beta}_{t}f .
\end{equation}
\label{eq:semi}  
\end{thm}

\begin{proof}
Using Fubinis Theorem, for "sufficiently good" function $f$, we have
\begin{align}
{}^\rho_{a} I^{\alpha}_{t}\,{}^\rho_{a} I^{\beta}_{t}f(x) 
       &=\frac{1}{(\rho + 1)^{\alpha -1}\Gamma({\alpha})} \int^x_a (x^{\rho +1} - \tau^{\rho +1})^{\alpha -1} \tau^\rho d\tau \nonumber\\
       &\hspace{2cm}\times \frac{1}{(\rho + 1)^{\beta -1}\Gamma({\beta})} \int^\tau_a (\tau^{\rho +1} - t^{\rho +1})^{\beta -1} t^\rho f(t)d\tau \nonumber \\
       &=\frac{1}{(\rho+1)^{\alpha +\beta - 2} \Gamma(\alpha)\Gamma(\beta)}\int_a^x t^\rho f(t) \nonumber \\
       & \hspace{2cm}\times \int_t^x \big(x^{\rho+1} - \tau^{\rho+1}\big)^{\alpha -1}\big(\tau^{\rho+1} - t^{\rho+1}\big)^{\beta-1}\, \tau^\rho \, d\tau \, dt \label{eq:pf1}\\
\nonumber\end{align}
The inner integral is evaluated by the change of variable $y = (\tau^{\rho+1}-t^{\rho+1})/(x^{\rho+1}-t^{\rho+1}):$
\begin{align}
\int_t^x \big(x^{\rho+1} - \tau^{\rho+1}\big)^{\alpha -1}\big(\tau^{\rho+1} - t^{\rho+1}&\big)^{\beta-1}\,\tau^\rho \, d\tau \, dt  \nonumber \\               &=\frac{(x^{\rho+1}-t^{\rho+1})}{\rho+1}\int_0^1(1-y)^{\alpha-1}y^{\beta -1}dy,\nonumber\\
&=\frac{(x^{\rho+1}-t^{\rho+1})}{\rho+1}B(\alpha,\beta)\nonumber\\
&=\frac{(x^{\rho+1}-t^{\rho+1})}{\rho+1}\cdot\frac{\Gamma(\alpha)\Gamma(\beta)}{\Gamma(\alpha+\beta)}\label{eq:pf2}
\end{align}
according to the known formulae for the beta function \cite{key-8,key-2,key-4}. Substituting (\ref{eq:pf2}) into (\ref{eq:pf1}) we obtain
\begin{align}
{}^\rho_{a} I^{\alpha}_{t}\,{}^\rho_{a} I^{\beta}_{t}f(x) 
               &=\frac{(\rho + 1)^{1-\alpha -\beta}}{\Gamma({\alpha+\beta})}\int^x_a (x^{\rho +1}- t^{\rho +1})^{\alpha +\beta-1}t^\rho f(t)dt,\nonumber\\
               &={}^\rho_{a} I^{\alpha+\beta}_{t}f(x), 
\end{align}
and thus, (\ref{eq:semi}) is proved for 'sufficiently good' functions $f$. 
If $\rho \geq c$ then by Theorem \ref{eq:th1} the operators $ {}^\rho_{a} I^{\alpha}_{t},\;{}^\rho_{a} I^{\beta}_{t}\, $ and ${}^\rho_{a} I^{\alpha+\beta}_{t}$ are bunded in $\textit{X}^p_c(a,b)$, hence the realation (\ref{eq:semi}) is true for $f \in \textit{X}^p_c(a,b).$
This completes the proof of the theorem \ref{eq:th2}.
\end{proof}

We have the following corollary.

\begin{cor}
Let $\alpha >0,\, \beta >0,\, 1 \leq p \leq \infty, \, 0 < a < b < \infty $ and let $\rho \in \mathbb{R}$ be such that $\rho \geq 1/p$. Then for $f \in \textit{L}^p(a,b)$ the semigroup property (\ref{eq:semi}) holds.
\end{cor}

We can extend (\ref{eq:fint}) to an arbitrary complex order $\alpha \in \mathbb{C}$ and can also define \emph{right-sided} generalized fractional integrals. 

\begin{dfn}
Let $\Omega = [a,b] \, (-\infty <a < b < \infty)$ be a finite interval on the real axis $\mathbb{R}$. The generalized fractional integral ${}^\rho I^\alpha_{a+}f$ of order $\alpha \in \mathbb{C} \; (Re(\alpha) > 0)$ is defined by
\begin{equation}
\big({}^\rho I^\alpha_{a+}f\big)(x) = \frac{(\rho + 1)^{1- \alpha }}{\Gamma({\alpha})} \int^x_a \frac{t^\rho f(t) }{(x^{\rho +1} - t^{\rho +1})^{1-\alpha}}\, dt 
\label{eq:df1}
\end{equation}
for $x > a$ and $Re(\alpha) > 0$. This integral is called the \emph{left-sided} fractional integral. Similarly we can define the \emph{right-sided} fractional integral
${}^\rho I^\alpha_{b-}f$ by
\begin{equation}
\big({}^\rho I^\alpha_{b-}f\big)(x) = \frac{(\rho + 1)^{1- \alpha }}{\Gamma({\alpha})} \int^b_x \frac{t^\rho f(t) }{(t^{\rho +1} - x^{\rho +1})^{1-\alpha}}\, dt
\label{eq:df2}
\end{equation}   
for $x < b$ and $Re(\alpha) > 0$. 
\end{dfn}

In the next section we will define \emph{generalized fractional derivatives} for an arbitrary complex order $\alpha \in \mathbb{C} \; (Re(\alpha) > 0)$ for Riemann-Liouville type (R-G) and similarly, we can define generalized fractional derivatives for Caputo-type (L-G). 


\section{Generalized fractional derivatives}
The integrals (\ref{eq:df1}) and (\ref{eq:df2}) allow one to define the corresponding generalized fractional derivatives. Here we only define Riemann-Liouville type derivatives. For $\rho \rightarrow -1^+$, we will also get the corresponding Hadamard-type derivatives as a natural generalization to the standard Hadamard fractional derivatives. The Caputo-type derivatives can be defined similarly. 


\begin{dfn}{( R-G Type )}

The corresponding \emph{Riemann-type fractinal derivatives}  ${}^\rho D^\alpha_{a+}f$ and ${}^\rho D^\alpha_{b-}f$ of order $\alpha \in \mathbb{C}, \; Re(\alpha)> 0$, are defined by 
\begin{equation}
\big({}^\rho D^{\alpha}_{a+}f\big)(x)  
              = \frac{(\rho + 1)^{\alpha -n +1}}{\Gamma({n -\alpha})}\, \frac{d^n}{dx^n} \int^x_a \frac{t^\rho f(t)}{(x^{\rho +1} - t^{\rho +1})^{\alpha -n +1}}\, dt, 
\label{eq:df3}              
\end{equation}
for $x > a$ and
\begin{equation}
\big({}^\rho D^{\alpha}_{b-}f\big)(x)  
              = \frac{(\rho + 1)^{\alpha -n +1}}{\Gamma({n -\alpha})}\, (-1)^n \frac{d^n}{dx^n} \int^b_x \frac{t^\rho f(t)}{(t^{\rho +1} - x^{\rho +1})^{\alpha -n+1}}\, dt 
\label{eq:df4}              
\end{equation}
for $x < b$, where $n = \lceil Re(\alpha)\rceil$. 
\end{dfn}

If $ 0 < Re(\alpha) < 1$, then 
\begin{equation}
\big({}^\rho D^{\alpha}_{a+}f\big)(x)  
              = \frac{(\rho + 1)^{\alpha -\lfloor Re(\alpha)\rfloor}}{\Gamma({1 -\alpha})}\, \frac{d}{dx} \int^x_a \frac{t^\rho f(t)}{(x^{\rho +1} - t^{\rho +1})^{\alpha -\lfloor Re(\alpha)\rfloor}}\, dt, 
\label{eq:df5}              
\end{equation}
for $x > a$ and
\begin{equation}
\big({}^\rho D^{\alpha}_{b-}f\big)(x)  
              = - \frac{(\rho + 1)^{\alpha -\lfloor Re(\alpha)\rfloor}}{\Gamma({1 -\alpha})}\, \frac{d}{dx} \int^b_x \frac{t^\rho f(t)}{(t^{\rho +1} - x^{\rho +1})^{\alpha -\lfloor Re(\alpha)\rfloor}}\, dt 
\label{eq:df6}              
\end{equation}
for $x<b$. 

When $\alpha \in \mathbb{R}^+$, formulars (\ref{eq:df3}) and (\ref{eq:df4}) take the following forms,
\begin{equation}
\big({}^\rho D^{\alpha}_{a+}f\big)(x)  
              = \frac{(\rho + 1)^{\alpha -n +1}}{\Gamma({n -\alpha})}\, \frac{d^n}{dx^n} \int^x_a \frac{t^\rho f(t)}{(x^{\rho +1} - t^{\rho +1})^{\alpha -n +1}}\, dt, 
\label{eq:df7}              
\end{equation}
for $x > a$ and
\begin{equation}
\big({}^\rho D^{\alpha}_{b-}f\big)(x)  
              = \frac{(\rho + 1)^{\alpha -n +1}}{\Gamma({n -\alpha})}\, (-1)^n \frac{d^n}{dx^n} \int^b_x \frac{t^\rho f(t)}{(t^{\rho +1} - x^{\rho +1})^{\alpha -n+1}}\, dt 
\label{eq:df8}              
\end{equation}
for $x < b$, where $n = \lceil \alpha \rceil$, while formulars (\ref{eq:df5}) and (\ref{eq:df6}) are given by
\begin{equation}
\big({}^\rho D^{\alpha}_{a+}f\big)(x)  
              = \frac{(\rho + 1)^{\alpha}}{\Gamma({1 -\alpha})}\, \frac{d}{dx} \int^x_a \frac{t^\rho f(t)}{(x^{\rho +1} - t^{\rho +1})^{\alpha}}\, dt, 
\label{eq:df9}              
\end{equation}
for $x > a$ and
\begin{equation}
\big({}^\rho D^{\alpha}_{b-}f\big)(x)  
              = - \frac{(\rho + 1)^{\alpha}}{\Gamma({1 -\alpha})}\, \frac{d}{dx} \int^b_x \frac{t^\rho f(t)}{(t^{\rho +1} - x^{\rho +1})^{\alpha}}\, dt 
\label{eq:df10}              
\end{equation}
for $x<b$. 

If $Re(\alpha) = 0 \,(\alpha \ne 0), $ then (\ref{eq:df5}) and (\ref{eq:df6}) give generalized fractional derivatives of a purely imaginary order,
\begin{equation}
\big({}^\rho D^{i\theta}_{a+}f\big)(x)  
              = \frac{(\rho + 1)^{i\theta}}{\Gamma({1 -i\theta})}\, \frac{d}{dx} \int^x_a \frac{t^\rho f(t)}{(x^{\rho +1} - t^{\rho +1})^{i\theta}}\, dt, 
\label{eq:df11}              
\end{equation}
for $x > a$ and
\begin{equation}
\big({}^\rho D^{i\theta}_{b-}f\big)(x)  
              = - \frac{(\rho + 1)^{i\theta}}{\Gamma({1 -i\theta})}\, \frac{d}{dx} \int^b_x \frac{t^\rho f(t)}{(t^{\rho +1} - x^{\rho +1})^{i\theta}}\, dt \quad ; x<b.
\label{eq:df12}              
\end{equation}
where $\theta \in \mathbb{R}$. 




Secondly, we give the definition for the genaralize fractional derivative for Caputo-type. 

\begin{dfn}{( C-G Type )}

Let $\alpha \in \mathbb{C}, \; Re(\alpha)> 0$ and $n = \lceil\alpha\rceil$. The generalised Caputo-type derivative of arbitrary order $\alpha$ of $f(x)$ is defined by
\begin{equation}
{}^\rho D^{\alpha}_{a+}f(x)  
              = \frac{(\rho + 1)^{\alpha -n +1}}{\Gamma({n -\alpha})}\, \int^x_a (x^{\rho +1} - \tau^{\rho +1})^{n -\alpha -1} \tau^\rho f^{(n)}(\tau) d\tau 
\label{eq:df1-3} 
\end{equation}
for $x > a$, and
\begin{equation}
{}^\rho D^{\alpha}_{b-}f(x)  
              = (-1)^n \frac{(\rho + 1)^{\alpha -n +1}}{\Gamma({n -\alpha})}\, \int^b_x (\tau^{\rho +1} - x^{\rho +1})^{n -\alpha -1} \tau^\rho f^{(n)}(\tau) d\tau 
\label{eq:df1-3-2} 
\end{equation}
for $x < b$, if the right-hand-sides exist.
\end{dfn} 
\noindent If $0 < \alpha < 1,$ (\ref{eq:df1-3}) reduces to
\begin{equation}
{}^\rho D^{\alpha}_{a+}f(x)  
              = \frac{(\rho + 1)^{\alpha}}{\Gamma({1 -\alpha})}\, \int^x_a \frac{\tau^\rho }{(x^{\rho +1} - \tau^{\rho +1})^{\alpha} } f^{'}(\tau) d\tau 
\label{eq:df1-4} 
\end{equation}


\noindent To demonstrate the use of the new derivative, we will obtain generalized fractional derivative of the power function for a special case. For simplicity assume $\alpha \in \mathbb{R}^+, \, 0 < \alpha < 1$ and $a=0$, and use the following version of the derivative
\begin{align}
  \big({}^\rho D^\alpha_{0+} f\big)(x) &= \frac{(\rho+1)^\alpha}{\Gamma(1-\alpha)} \, \frac{d}{dx}\int^x_0 \frac{t^\rho}{(x^{\rho+1}-t^{\rho+1})^\alpha} f(t)dt 
\label{eq:simple0}
\end{align}

\begin{ex}
We find the generalized derivative of the function $f(x) = x^\nu$, where $\nu \in \mathbb{R}$. The formula (\ref{eq:simple0}) yields
\begin{align}
  {}^\rho D^\alpha_{0+} x^\nu = \frac{(\rho+1)^\alpha}{\Gamma(1-\alpha)} \, \frac{d}{dx}\int^x_0 \frac{t^\rho}{(x^{\rho+1}-t^{\rho+1})^\alpha}\, t^\nu \, dt
\label{eq:simple1}  
\end{align}
To evaluate the inner integral, we use the substitution $u=t^{\rho+1}/x^{\rho+1}$ to obtain,
\begin{align}
   \int^x_0 \frac{t^\rho}{(x^{\rho+1}-t^{\rho+1})^\alpha}\, t^\nu \, dt 
       &= \frac{x^{(\rho+1)(1-\alpha)+\nu}}{\rho+1} \int^1_0 \frac{u^\frac{\nu}{1+\rho}}{(1-u)^\alpha} \, du \nonumber \\
       &=\frac{x^{(\rho+1)(1-\alpha)+\nu}}{\rho+1} \int^1_0 u^{\frac{\nu+\rho+1}{\rho+1}-1}(1-u)^{(1-\alpha)-1} \, du \nonumber \\
       &= \frac{x^{(\rho+1)(1-\alpha)+\nu}}{\rho+1}\, B\Big(1-\alpha, \frac{\nu+\rho+1}{\rho+1}\Big) \label{eq:simple2}
\end{align}
where $B(.\, ,.)$ is the Beta function. Combining (\ref{eq:simple2}) with (\ref{eq:simple1}), we obtain,
\begin{align}
    {}^\rho D^\alpha_{0+} x^\nu &= \frac{(\rho+1)^{\alpha-1}}{\Gamma(1-\alpha)} \, \frac{d}{dx} \, x^{(\rho+1)(1-\alpha)+\nu}\,B\Big(1-\alpha, \frac{\nu+\rho+1}{\rho+1}\Big) \nonumber\\
    &= \frac{(\rho+1)^{\alpha-1}\Gamma\Big(\frac{\nu}{\rho+1}+1\Big)}{\Gamma\Big(\frac{\nu}{\rho+1}+1-\alpha\Big)}\, x^{\nu+(\rho+1)(1-\alpha)-1} 
\label{eq:gen}    
\end{align}
for $\rho > -1$, after using the properties of the Beta function \cite{key-4} and the relation $\Gamma(z+1)=z\,\Gamma(z)$. When $\rho = 0$ we obtain the Riemann-Liouville fractional derivative of the power function given by \cite{key-2, key-8, key-9},
\begin{equation}
  D^\alpha_{0+} x^\nu = \frac{\Gamma\Big(\nu+1\Big)}{\Gamma\Big(\nu+1-\alpha\Big)}\; x^{\nu-\alpha} 
\end{equation}
\end{ex}
To compare results, we plot (\ref{eq:gen}) for several values of $\rho$. We also consider different values of $\nu$ to see the effect on the degree of the power function. The results are summaries in figure 1 and figure 2:
\begin{figure}[h]
	\centering
	  \subfloat[$\nu$= 1.0]{\includegraphics[width=2.2in, height=3.0in]{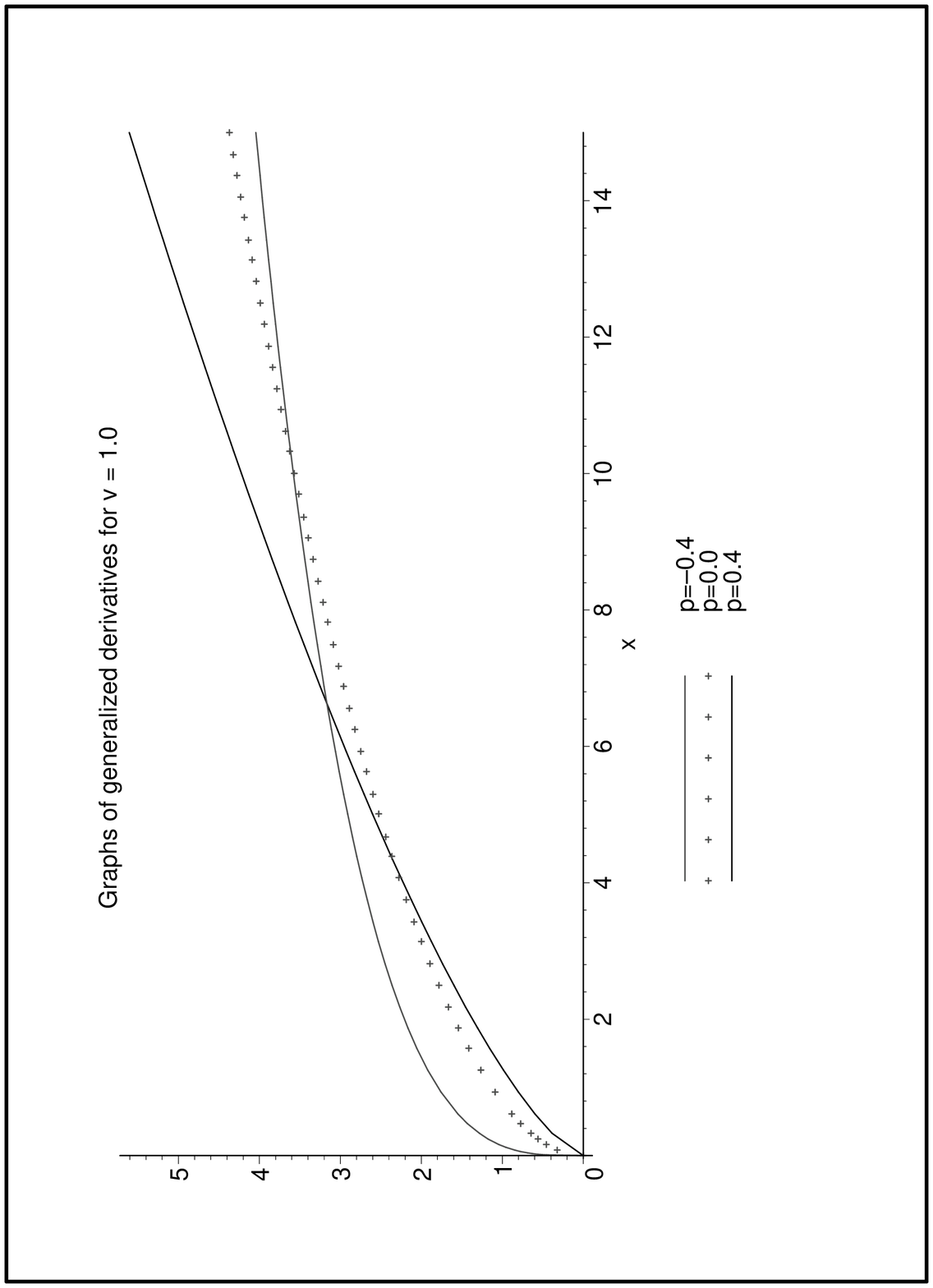}}
	  \hspace{0.1in}
	  \subfloat[$\nu$= 0.2]{\includegraphics[width=2.2in, height=3.0in]{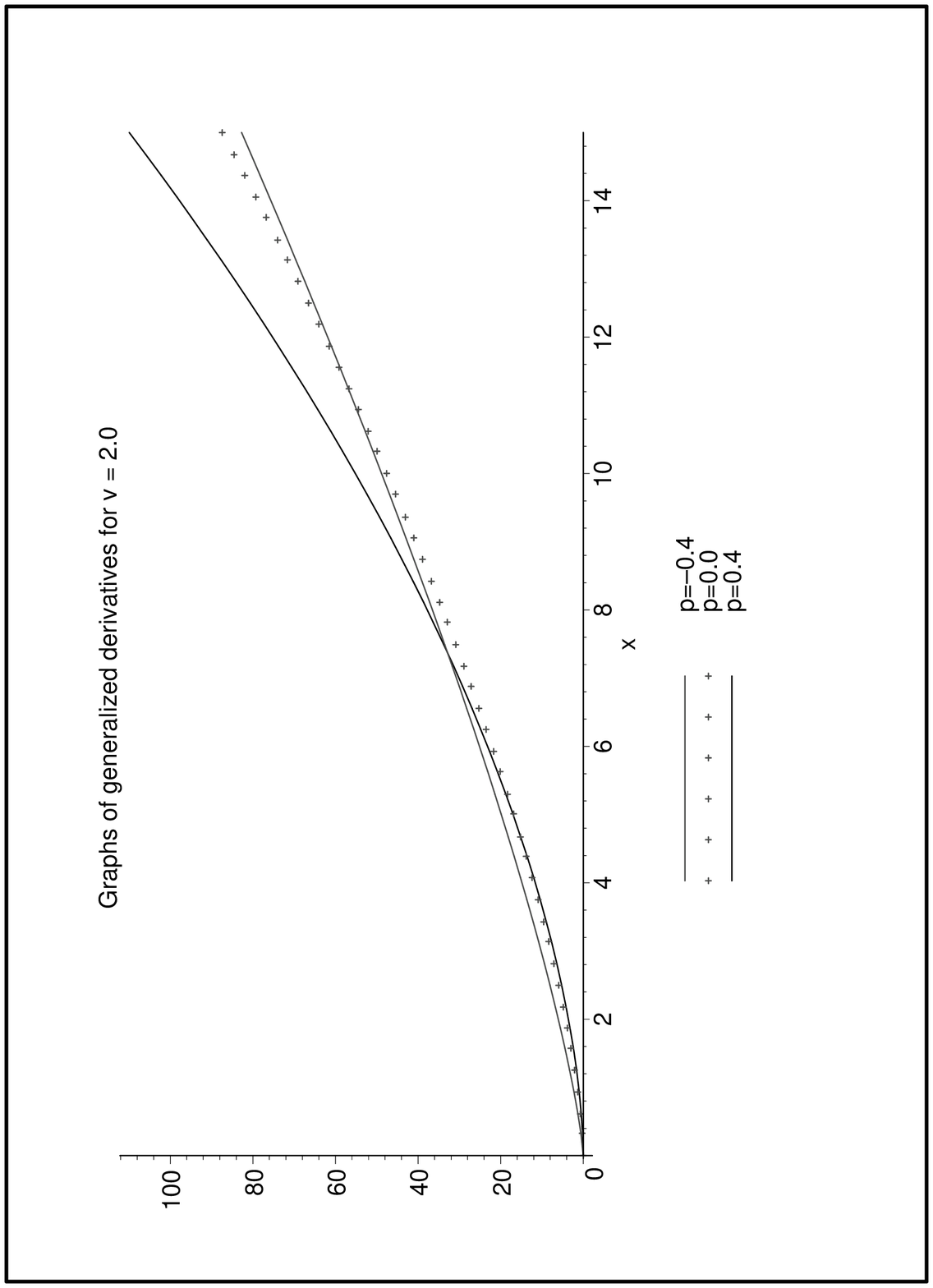}}\\
	  \caption{Generalized fractional derivative of the power function $x^\nu$ for $\rho = -0.4,\, 0.0,\, 0.4$ and $\nu = 1.0, \, 2.0$}        
	\label{fig:FD-1}
\end{figure}

\begin{figure}[h]
	\centering
	  \subfloat[$\nu$= 0.5]{\includegraphics[width=2.2in, height=3.0in]{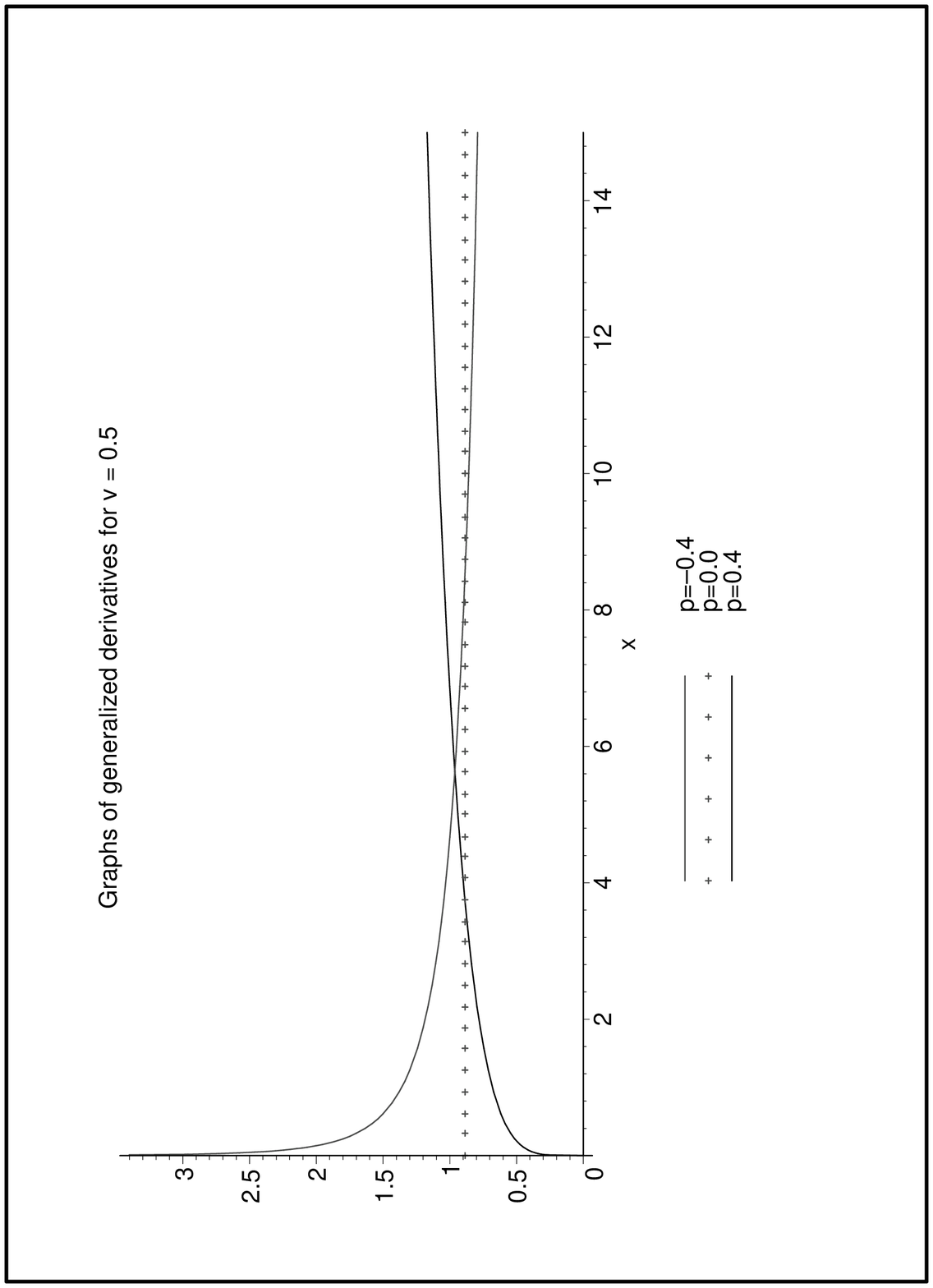}}
	  \hspace{0.1in}
	  \subfloat[$\nu$= 1.5]{\includegraphics[width=2.2in, height=3.0in]{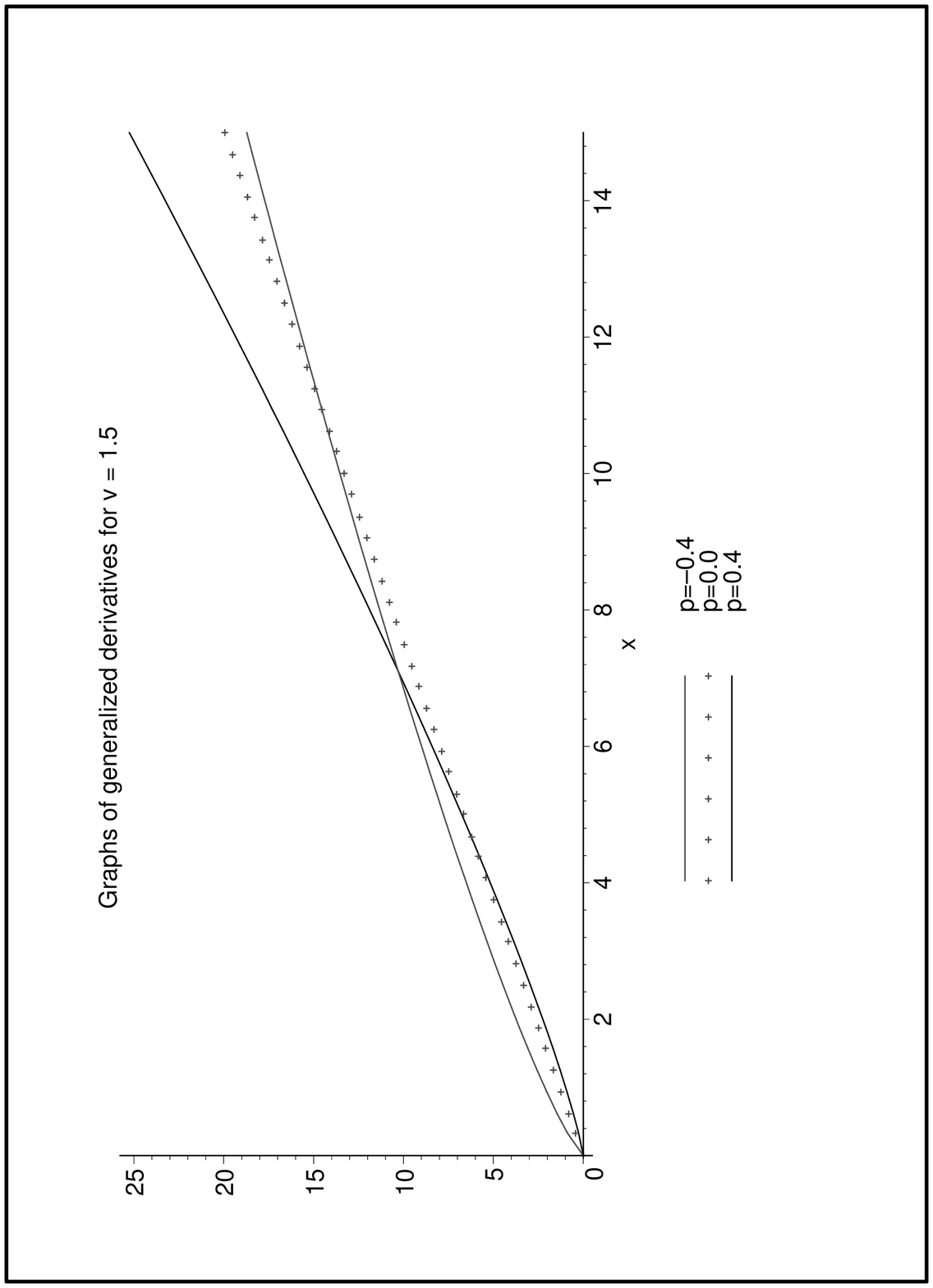}}
	\caption{Generalized fractional derivative of the power function $x^\nu$ for $\rho = -0.4,\, 0.0,\, 0.4$ and $\nu = 0.5, \, 1.5$}        
	\label{fig:FD-2}
\end{figure}

We conclude the paper with the following open problem.
\begin{prb}
Find an exact formula for the left-sided generalized fractional derivative of the power function, $(x-a)^\omega$ with $\omega \in \mathbb{R}$, that is, evaluate the following integral,
\begin{equation}
{}^\rho D^{\alpha}_{a+}(x-a)^\omega  
              = \frac{(\rho + 1)^{\alpha -n +1}}{\Gamma({n -\alpha})}\, \frac{d^n}{dx^n} \int^x_a \frac{t^\rho (t-a)^\omega}{(x^{\rho +1} - t^{\rho +1})^{\alpha -n +1}}\, dt, 
\label{eq:o-prb}              
\end{equation}
for $x > a$ where $\alpha \in \mathbb{C}$, $n = \lceil Re(\alpha)\rceil$ and $\rho \ne -1$.
\end{prb}

\begin{con}
According to the figure 1 and figure 2, we notice that the characteristics of the fractional derivative is highle affected by the value of $\rho$, thus it provides a new direction for the control applications. 

The paper presents a new fractional integration, which generalizes the Riemann-Liouville and Hadamard fractional integrals into a single form, which when a parameter fixed at different values, produces the above integrals as special cases. Conditions are given for such a generalized fractional integration operator to be bounded in an extended Lebesgue measurable space. Semigroup property for the above operator is also proved. We also gave a general definition of the Fractional derivatives and gave some examples.

In a future paper, we will derive formulae for the Laplace, Fourier and Mellin transforms for the generalized fractional integral. We already know that we can deduce Hadamard and Riemann-Liouville integrals for the special cases of $\rho$. We want to further investigate the effect on the new parameter $\rho$. We will also study the generalize fractional derivatives and their properties. Those results will appear elsewhere. 
\end{con}



%

\end{document}